\documentclass[a4paper]{amsart}
\usepackage[utf8]{inputenc}
\usepackage{amsthm, amsmath, amsfonts, amssymb, commath, cancel, url, hyperref, cleveref, xcolor}
\usepackage[backend=biber,style=alphabetic]{biblatex}
\usepackage{enumitem}
\setlist[enumerate,1]{label={(\roman*)}}

\numberwithin{equation}{section}
\theoremstyle{remark}
\newtheorem{rmk}{Remark}
\numberwithin{rmk}{section}
\theoremstyle{plain}
\newtheorem{lem}[rmk]{Lemma}
\newtheorem{prop}[rmk]{Proposition}
\newtheorem{thm}[rmk]{Theorem}

\newtheorem{theorem}{Theorem}

\newcommand{\nat}{\mathbb{N}}
\newcommand{\real}{\mathbb{R}}

\title[Least energy solutions to $(\mathrm{SBP}_{a,\rho})$]{Existence and limit behavior of least energy solutions to constrained Schrödinger--Bopp--Podolsky systems in $\real^3$}
\author{Gustavo de Paula Ramos}
\author{Gaetano Siciliano}
\address{Instituto de Matem\'atica e Estat\'istica\\
Universidade de S\~ao Paulo\\
Rua do Mat\~ao, 1010\\
05508-090\\
S\~ao Paulo, SP\\
Brazil}
\email{gustavopramos@gmail.com, gpramos@ime.usp.br}
\urladdr{http://gpramos.com}
\email{sicilian@ime.usp.br}
\begin{document}

\begin{abstract}
Consider the following Schrödinger--Bopp--Podolsky system in $\real^3$ under an $L^2$-norm constraint,
\[
	\begin{cases}
	-\Delta u + \omega u + \phi u = u|u|^{p-2},\\
	-\Delta \phi + a^2\Delta^2\phi=4\pi u^2,\\
	\|u\|_{L^2}=\rho,
	\end{cases}
\]
where $a,\rho>0$ and our unknowns are $u,\phi\colon\real^3\to\real$ and $\omega\in\real$. We prove that if $2<p<3$ (resp., $3<p<10/3$) and $\rho>0$ is sufficiently small (resp., sufficiently large), then this system admits a least energy solution. Moreover, we prove that if $2<p<14/5$ and $\rho>0$ is sufficiently small, then least energy solutions are radially symmetric up to translation and as $a\to 0$, they converge to a least energy solution of the Schrödinger--Poisson--Slater system under the same $L^2$-norm constraint.

\smallskip
\noindent \textbf{Keywords.} Elliptic systems; Schrödinger--Bopp--Podolsky equations; Constrained Minimization; Critical Point Theory.
	
\smallskip
\noindent \textbf{2010 Mathematics Subject Classification.} 35B38, 35A15, 35Q55.
\end{abstract}

\date{\today}
\maketitle

\section{Context and main results}
We are interested in the following Schrödinger--Bopp--Podolsky system under a constraint,
\begin{equation}\label{intro:eqn:SBP_a,rho}\tag{$\text{SBP}_{a,\rho}$}
	\begin{cases}
		-\Delta u + \omega u + \phi u = u |u|^{p-2},\\
		-\Delta\phi+a^2\Delta^2\phi=4\pi u^2,\\
		\|u\|_{L^2}=\rho,
	\end{cases}
\end{equation}
where $a,\rho>0$, $p\in ]2,10/3[\setminus\{3\}$ are fixed parameters and $u,\phi\colon\real^3\to\real$, $\omega \in \real$ are our unknowns.

The system of equations \eqref{intro:eqn:SBP_a,rho} is obtained when we search for standing wave solutions to the Schrödinger equation while considering the electromagnetic interaction term from Bopp--Podolsky theory for a unitary electrical charge. In fact, the reader may find a more detailed deduction of such system in \cite[Section 2]{dAvenia_2019}.

Since d'Avenia and Siciliano's 2019 article, the Schrödinger--Bopp--Podolsky system has become the object of research in a number of recent works. In fact, the Neumann problem for the Schrödinger--Bopp--Podolsky under an $L^2$-norm constraint was considered in Afonso and Siciliano's 2021 article \cite{Afonso_2021}, where the Kranoselskii genus was used to prove that, under certain conditions on the boundary, the Schrödinger--Bopp--Podolsky system in a bounded domain with nonconstant coupling admits infinite solutions. A setting similar to Afonso and Siciliano's was considered in Soriano Hernandez' thesis \cite{Soriano_2021}, where genus theory was applied on the corresponding Dirichlet problem to analogously deduce the existence of an infinite number of solutions.

In this context, the goal of this note is to show that some well known facts about the Schrödinger--Poisson--Slater system under an $L^2$-norm constraint,
\begin{equation}\label{intro:eqn:SPS_rho}\tag{$\text{SPS}_\rho$}
\begin{cases}
	-\Delta u + \omega u + \phi u = u|u|^{p-2},\\
	-\Delta\phi=4\pi u^2,\\
	\|u\|_{L^2}=\rho,
\end{cases}
\end{equation}
also hold for \eqref{intro:eqn:SBP_a,rho}.

Technically, we will consider weak solutions to \eqref{intro:eqn:SBP_a,rho} (resp., to \eqref{intro:eqn:SPS_rho}) in the space $H^1(\real^3)\times\mathcal{X}(\real^3)$ (resp., in $H^1(\real^3)\times\mathcal{D}^{1,2}(\real^3)$), where the spaces $\mathcal{D}^{1,2}(\real^3)$, $H^1(\real^3)$ and $\mathcal{X}(\real^3)$ are obtained as the respective Hilbert space completions of $C_c^\infty(\real^3)$ with respect to the following inner products:
\[
	\langle u,v \rangle_{\mathcal{D}^{1,2}}
	:=
	\int \langle\nabla u,\nabla v\rangle,
\]
\[
	\langle u,v \rangle_{H^1}
	:=
	\int \langle\nabla u,\nabla v\rangle + uv
\]
and
\[
	\langle u,v \rangle_{\mathcal{X}}
	:=
	\int
		\langle\nabla u, \nabla v\rangle+\Delta u \Delta v.
\]
We remark that it is proved in \cite[Appendix A.1]{dAvenia_2019} that if $(u,\phi,\omega)\in H^1(\real^3)\times\mathcal{X}(\real^3)\times\real$ is a weak solution to \eqref{intro:eqn:SBP_a,rho}, then $u\in C^2(\real^3)$ and $\phi\in C^4(\real^3)$. For this reason, we will state our results in terms of \emph{solutions} instead of \emph{weak solutions}.

In Section \ref{charac}, we will define $\mathcal{J}_a\in C^1(H^1(\real^3))$ such that if $u\in H^1(\real^3)$ is a critical point of $\left.\mathcal{J}_a\right|_{S_\rho}$ and $\omega\in\real$ is its associated Lagrange multiplier, then there exists a unique $\phi \in \mathcal{X}(\real^3)$ such that $(u,\phi,\omega)$ is a solution to \eqref{intro:eqn:SBP_a,rho}, where
\[
	S_\rho:=\{u \in H^1(\real^3): \|u\|_{L^2}=\rho\}.
\]
With that in mind, our problem becomes that of searching for critical points of $\left.\mathcal{J}_a\right|_{S_\rho}$. Furthermore, if $(u,\phi,\omega)\in H^1(\real^3)\times\mathcal{X}(\real^3)\times\real$ is a solution to \eqref{intro:eqn:SBP_a,rho} such that
\[
	\mathcal{J}_{a,\rho}:=\inf_{v\in S_\rho}\mathcal{J}_a(v)=\mathcal{J}_a(u),
\]
then we will call it a \emph{least energy solution}.

In this situation, our first result guarantees the existence of least energy solutions whenever $2<p<3$ and $\rho>0$ is sufficiently small:
\begin{theorem}\label{intro:theorem:small_rho}
If $2<p<3$, then there exists $R_p>0$ such that given $a>0$ and $\rho\in ]0,R_p[$, the system \eqref{intro:eqn:SBP_a,rho} admits a least energy solution.
\end{theorem}

Similar arguments suffice to obtain least energy solutions whenever $3<p<10/3$ and $\rho>0$ is sufficiently large:
\begin{theorem}\label{intro:theorem:large_rho}
If $3<p<10/3$, then there exists $R_p>0$ such that given $a>0$ and $\rho>R_p$, the system \eqref{intro:eqn:SBP_a,rho} admits a least energy solution.
\end{theorem}

Now, let us we focus on the case $2<p<14/5$. Arguments similar to those in \cite{Georgiev_2012} allow us to obtain the radiality of least energy solutions when $\rho>0$ is sufficiently small:
\begin{theorem}\label{intro:theorem:radiality}
Given $p \in ]2,14/5[$, there exists $R_p>0$ such that if $a>0$, $\rho\in]0,R_p[$ and $(u,\phi,\omega) \in S_\rho\times\mathcal{X}(\real^3)\times\real$ is a least energy solution to \eqref{intro:eqn:SBP_a,rho}, then $u$ is radially symmetric up to translation.
\end{theorem}

Our last result is that if $\rho>0$ is sufficiently small, then a family of least energy solutions to the Schrödinger--Bopp--Podolsky system (indexed by $a>0$) converges to a least energy solution to the Schrödinger--Poisson--Slater system as $a \to 0$:
\begin{theorem}\label{intro:theorem:a-to-zero:1}
If $2<p<14/5$, then there exists $R_p>0$ such that given $\rho\in ]0,R_p[$ and a set
\begin{multline}\label{intro:eqn:a-to-zero}
	\{
		(u_a,\phi_a,\omega_a) \in H^1(\real^3)\times \mathcal{X}(\real^3)\times\real: a>0
		~\text{and}~
		\\
		(u_a,\phi_a,\omega_a)
		~\text{is a least energy solution to}~
		\eqref{intro:eqn:SBP_a,rho}
	\},
\end{multline}
we conclude that \eqref{intro:eqn:SPS_rho} admits a least energy solution $(u_0,\phi_0,\omega_0)\in H^1_{\mathrm{rad}}(\real^3)\times \mathcal{D}^{1,2}_{\mathrm{rad}}(\real^3)\times ]0,\infty[$ such that, up to translations and subsequences,
\[
	(u_a,\phi_a,\omega_a)
	\to
	(u_0,\phi_0,\omega_0)
	~\text{in}~
	H^1(\real^3)\times\mathcal{D}^{1,2}(\real^3)\times\real
	~\text{as}~a\to 0.
\]
\end{theorem}

We remark that Theorem \ref{intro:theorem:small_rho} guarantees that if $\rho>0$ is sufficiently small, then $H^1(\real^3)\times\mathcal{X}(\real^3)\times\real$ admits a subset of the form \eqref{intro:eqn:a-to-zero}.

Let us comment on our main results and their proofs. Theorems \ref{intro:theorem:small_rho} and \ref{intro:theorem:large_rho} are respectively analogous to \cite[Theorem 4.1]{Bellazzini_2011} and \cite[Theorem 1.1]{Bellazzini_2010}, where Bellazzini and Siciliano have instead considered the Schrödinger--Poisson--Slater system. The proofs of Theorems \ref{intro:theorem:small_rho} and \ref{intro:theorem:large_rho} follow the same strategy: if we could prove that minimizing sequences for $\left.\mathcal{J}_a\right|_{S_\rho}$ are precompact, then we would conclude that the infimum $\inf_{v\in S_\rho} \mathcal{J}_a(v)$ is achieved in $S_\rho$. For this purpose, we follow the abstract framework for constrained minimization problems developed by Bellazzini and Siciliano in \cite{Bellazzini_2010,Bellazzini_2011}.

Theorem \ref{intro:theorem:radiality} follows from slight changes to the general argument in \cite{Georgiev_2012}. We remark that Georgiev, Prinari and Visciglia's argument is based on the study of the behavior of least energy solutions under an appropriate rescaling. In a certain sense, this rescaling allows one to approach the behavior of least energy solutions to \eqref{intro:eqn:SPS_rho} as $\rho\to 0$ by solutions to the semilinear partial differential equation
\[
	-\Delta u + \Omega u = u|u|^{p-2},
\]
where $\Omega>0$ is a certain constant. Due to the behavior of the energy functional $\mathcal{J}_a$ under rescalings (see Lemma \ref{prior:lem:rescalings}), the arguments of Georgiev, Prinari and Visciglia may only by applied in our current context under the technical hypothesis $2<p<14/5$.

The result \cite[Theorem 1.3]{dAvenia_2019} was the source of inspiration for Theorem \ref{intro:theorem:a-to-zero:1}. A nontrivial aspect of Theorem \ref{intro:theorem:a-to-zero:1} is the fact that the Lagrange multiplier $\omega_0$ is positive, which follows from \cite{Georgiev_2012} in the case when $2<p<14/5$ and $\rho>0$ is sufficiently small (see Proposition \ref{multipliers:prop}). Indeed, d'Avenia and Siciliano have instead considered the following unconstrained Schrödinger--Bopp--Podolsky system in their original result:
\[
	\begin{cases}
		-\Delta u + \omega u + q^2\phi u = u |u|^{p-2},\\
		-\Delta\phi+a^2\Delta^2\phi=4\pi u^2,
	\end{cases}
\]
where $\omega>0$ is a fixed parameter. Their proof relies on the fact that the mapping $R_\omega\colon H^1_{\mathrm{rad}}(\real^3)\to H^{-1}_{\mathrm{rad}}(\real^3)$ given by
\[
	R_\omega u=\langle\nabla u,\nabla \cdot\rangle_{L^2}+\omega\langle u,\cdot\rangle_{L^2}.
\]
is a Riesz isomorphism precisely because $\omega>0$. In this context, we present two arguments for convergence: the first relies on the abstract minimization method by Bellazzini and Siciliano, and the second relies on Riesz isomorphisms of the aforementioned form.

\subsection*{Notation}
Unless denoted otherwise, $\real^3$ is the implied domain of integration.

The subscript $\mathrm{rad}$ means that the corresponding Hilbert space is obtained as a completion of
\[
	C_{c,\mathrm{rad}}^\infty(\real^3):=\{u \in C_c^\infty(\real^3): u(x)=u(y)~\text{whenever}~|x|=|y|\}
\]
instead of a completion of $C_c^\infty(\real^3)$.

Rescalings are a central tool in various results of this note. For this reason, we fix the following notation: given $u\in H^1(\real^3)$, $\beta\in\real$ and $\theta>0$, we define $u_{\beta,\theta}\in H^1(\real^3)$ as the rescaling
\[
	u_{\beta,\theta}(x)=\theta^{1-3\beta/2} u(\theta^{-\beta}x)
	~\text{for a.e.}~x \in \real^3,
\]
so that $\|u_{\beta,\theta}\|_{L^2}=\theta\|u\|_{L^2}$.

\subsection*{Acknowledgments}
During the development of this article, G. de Paula Ramos was fully supported by Capes grant 88887.614697/2021-00 and G. Siciliano was partially supported by Fapesp grant 2019/27491-0, CNPq grant 304660/2018-3, FAPDF, and CAPES (Brazil) and
INdAM (Italy).

\section{Preliminaries}\label{prelim}
In this section, we present preliminary results which are relevant for our arguments. Only a few of these are proved here because most of these correspond to well documented facts.

\subsection{Characterizing solutions to \eqref{intro:eqn:SPS_rho}, \eqref{intro:eqn:SBP_a,rho}}\label{charac}
Fix $u \in H^1(\real^3)$. It is well known that the partial differential equation
\[
	-\Delta \phi=4\pi u^2
\]
admits a unique weak solution in $\mathcal{D}^{1,2}(\real^3)$. Specifically, it is given by
\[
	\phi^u_0:=u^2*|\cdot|^{-1},
\]
where $*$ denotes the convolution.

Now, fix $a>0$. D'Avenia and Siciliano prove in \cite[Section 3.1]{dAvenia_2019} the following analogue result: the partial differential equation
\[
	-\Delta \phi + a^2\Delta^2\phi=4\pi u^2
\]
admits a unique solution in $\mathcal{X}(\real^3)$. Specifically, it is given by
\[
	\phi^u_a:=u^2*\kappa_a,
\]
where $\kappa_a\colon\real^3\setminus\{0\}\to ]0,1/a[$ is given by
\[
	\kappa_a(x)=\frac{1-\exp(-|x|/a)}{|x|}.
\]

Consider $(u,\phi,\omega) \in H^1(\real^3)\times\mathcal{X}(\real^3)\times\real$. In view of the previous paragraph, we obtain the following equivalence: $(u,\varphi,\omega)$ is a solution to \eqref{intro:eqn:SBP_a,rho} if, and only if, $\varphi=\phi^u_a$ and $(u,\omega)$ satisfies
\begin{equation}\label{charac:eqn:constrained}
	\begin{cases}
		-\Delta u + \omega u + \phi_a^u u = u|u|^{p-2},\\
		\|u\|_{L^2}=\rho.
	\end{cases}
\end{equation}

In this context, we define the \emph{energy functional} $\mathcal{J}_a\colon H^1(\real^3)\to \real$ as
\[
	\mathcal{J}_a(u)
	=
	\frac{1}{2}\|\nabla u\|_{L^2}^2
	+
	\frac{1}{4}\int \phi^u_a u^2
	-
	\frac{1}{p}\|u\|_{L^p}^p.
\]
Let us prove that $\mathcal{J}_a\colon H^1(\real^3)\to\real$ is well defined. Due to the inequality $\kappa_a\leq |\cdot|^{-1}$, the following is an implication of the Hardy--Littlewood--Sobolev inequality (\cite[Section 4.3]{Lieb_Loss_2001}):
\begin{multline}\label{energy-functional:eqn}
	\text{there exists}~K>0~\text{such that}
	\\
	\int \phi_a^uu^2
	\leq
	\int \phi_0^uu^2
	\leq
	K\|u\|_{L^{12/5}}^4
	~\text{for every}~u\in L^{12/5}(\real^3).
\end{multline}
The conclusion then follows from the Sobolev embeddings $H^1(\real^3)\hookrightarrow L^p(\real^3)$, $L^{12/5}(\real^3)$.

We can employ usual arguments to prove that $\mathcal{J}_a\colon H^1(\real^3)\to \real$ is a functional of class $C^1$. With such energy functional, we obtain the following variational characterization of solutions to \eqref{charac:eqn:constrained}: $(u,\omega)\in H^1(\real^3)\times\real$ solves \eqref{charac:eqn:constrained} precisely when $u$ is a critical point of $\left.\mathcal{J}_a\right|_{S_\rho}$ and $\omega$ is its respective Lagrange multiplier, i. e.,
\[
	\mathrm{d}(\mathcal{J}_a)_u=-\omega\langle u,\cdot \rangle_{L^2}\in H^{-1}(\real^3).
\]
We likewise define the energy functional $\mathcal{J}_0\colon H^1(\real^3)\to\real$ and obtain a variational characterization of solutions to \eqref{intro:eqn:SPS_rho}.

We finish this section by stating a few facts about the energy functional: the following inequality will be frequently used:
\begin{equation}\label{charac:ineq}
	\mathcal{J}_a(u)\leq\mathcal{J}_0(u)~\text{for every}~a\geq 0~\text{and}~u\in H^1(\real^3),
\end{equation}
which follows from the fact that $\kappa_a\leq|\cdot|^{-1}$.

It suffices to argue precisely as in \cite[Lemma 3.1]{Bellazzini_2010} to prove that
\begin{lem}\label{charac:lem}
Given $\rho>0$, there exists $C_\rho>0$ such that
\[
	\mathcal{J}_a(u)+C_\rho\rho^2\geq C_\rho\|u\|_{H^1}^2
\]
for every $a\geq 0$ and $u \in S_\rho$.
\end{lem}

If $2<p<3$, then the Lagrange multipliers of solutions to
\eqref{intro:eqn:SPS_rho} tend to a certain $\Omega>0$ as $\rho\to 0$:
\begin{prop}[{\cite[Proposition 1.3]{Georgiev_2012}}]\label{multipliers:prop}
If $2<p<3$, then there exists $\Omega>0$ such that
\[
	\sup_{\omega\in\mathcal{A}_\rho} |\omega-\Omega| \to 0
	~\text{as}~
	\rho\to 0,
\]
where given $\rho>0$,
\begin{multline*}
	\mathcal{A}_\rho:=
	\{
		\omega \in \real:
		~\text{there exists}~
		u \in H^1(\real^3)
		~\text{such that}~
		\\
		(u,\phi_0^u,\omega)\in H^1(\real^3)\times\mathcal{D}^{1,2}(\real^3)\times\real
		~\text{is a least energy solution to}~
		\eqref{intro:eqn:SPS_rho}
	\}.
\end{multline*}
\end{prop}

The following result explicits an upper bound for the convolution term of the energy functional $\mathcal{J}_a\colon H^1(\real^3)\to\real$ when we consider rescalings of functions in $H^1(\real^3)$:
\begin{lem}\label{prior:lem:rescalings}
Let $a>0$, $u\in H^1(\real^3)$, $\beta\in\real$ and $\theta>0$. It holds that
\[
	\int \phi_a^{u_{\beta,\theta}}u_{\beta,\theta}^2
	=
	\theta^4
	\int[u_{\beta,\theta}^2*\kappa_a(\theta^\beta \cdot)](x)u(x)^2\mathrm{d} x
	\leq
	\theta^{4-\beta}\int \phi_0^u u^2.
\]
\end{lem}
\begin{proof}
Let us begin by proving that
\[
	\int \phi_a^{u_{\beta,\theta}}u_{\beta,\theta}^2
	=
	\theta^4
	\int
		[u^2*\kappa_a(\theta^\beta \cdot)](x)u(x)^2
	\mathrm{d} x.
\]
By definition,
\[
	\int \phi_a^{u_{\beta,\theta}}u_{\beta,\theta}^2
	=
	\theta^{2-3\beta}\int \phi_a^{u_{\beta,\theta}}(x)u(\theta^{-\beta}x)^2\mathrm{d} x.
\]
A variable change yields
\begin{equation}\label{prelim:lem:eqn:1}
	\int \phi_a^{u_{\beta,\theta}}u_{\beta,\theta}^2
	=
	\theta^2\int \phi_a^{u_{\beta,\theta}}(\theta^\beta y)u(y)^2\mathrm{d} y.
\end{equation}

Fix $y \in \real^3$. By definition,
\[
	\phi_a^{u_{\beta,\theta}}(\theta^\beta y)
	=
	\int \kappa_a(x-\theta^\beta y)u_{\beta,\theta}(x)^2\mathrm{d} x
	=
	\theta^{2-3\beta}\int \kappa_a(x-\theta^\beta y)u(\theta^{-\beta}x)^2\mathrm{d} x.
\]
A variable change shows that
\[
	\phi_a^{u_{\beta,\theta}}(\theta^\beta y)
	=
	\theta^2
	\int \kappa_a(\theta^\beta(z-y))u(z)^2\mathrm{d} z.
\]
Finally, the result follows from \eqref{prelim:lem:eqn:1}, the previous equality and the elementary inequality $\kappa_a\leq |\cdot|^{-1}$.
\end{proof}

It follows from elementary computations that
\begin{lem}\label{preliminary_results:lem:2}
Let $u \in H^1(\real^3)$ and $\beta\in\real$. If $a>0$ and we take $I=\mathcal{J}_a$, then the function $h^u_\beta\colon]0,\infty[\to\real$ defined as $h^u_\beta(\theta)=I(u_{\beta,\theta})-\theta^2I(u)$ satisfies
\begin{multline*}
	h^u_\beta(\theta)
	=
	\frac{1}{2}(\theta^{2-2\beta}-\theta^2)\|\nabla u\|_{L^2}^2
	+
	\\
	+
	\frac{1}{4}
	\left[
	\theta^4\int\int \kappa_a(\theta^{\beta}(y-x))u(y)^2u(x)^2 \mathrm{d} x \mathrm{d} y
	-
	\theta^2\int \phi_u u^2
	\right]
	+
	\\
	-
	\frac{1}{p}[\theta^{(1-3\beta/2)p+3\beta-2}-\theta^2]\|u\|_{L^p}^p
\end{multline*}
and
\begin{multline*}
	(h^u_\beta)'(1)
	=
	-\beta\|\nabla u\|_{L^2}^2
	+
	\\
	+
	\frac{1}{4}
	\left[
		\frac{\beta}{a}
		\int \int e^{-|y-x|/a}u(y)^2u(x)^2 \mathrm{d} x \mathrm{d} y
		+
		(2-\beta)\int \phi_u u^2
	\right]
	+
	\\
	-
	\frac{1}{p}
	\left[
		\left(
			1-\frac{3}{2}\beta
		\right)
		p+3\beta-2
	\right]
	\|u\|_{L^p}^p.
\end{multline*}
\end{lem}

\subsection{The abstract minimization problem}
In \cite{Bellazzini_2010,Bellazzini_2011}, Bellazzini and Siciliano developed an abstract framework to prove the existence of solutions to constrained minimization problems. We shall follow this framework in our proofs of Theorems \ref{intro:theorem:small_rho} and \ref{intro:theorem:large_rho}, so we begin by sketching a restricted version of their results which may be readily applied to the considered problem.

Let $T\colon H^1(\real^3)\to\real$ be a functional and define $I\colon H^1(\real^3)\to\real$ as
\[
	I(u)=\frac{1}{2}\|\nabla u\|_{L^2}^2+T(u).
\]
With these functionals in mind, we are interested in the constrained minimization problem
\[
	I_\rho:=\inf_{u\in S_\rho} I(u),
\]
for $\rho\geq 0$, where we define $I_0=0$.

In this context, the following results were developed by Bellazzini and Siciliano as to congregate sufficient conditions to avoid dichotomy when applying Lions' concentration-compactness principle. The first one is a corollary of \cite[Theorem 2.1]{Bellazzini_2011}:
\begin{thm}\label{proofs:thm:2.1}
Suppose that
\begin{multline}\label{proofs:eqn:I-0}
	\text{if}~
	\rho>0
	~\text{and}~
	(u_n)_{n \in \nat}
	~\text{is a minimizing sequence for}~I_\rho,
	~\text{then there exists}
	\\
	u_\infty\in H^1(\real^3)\setminus\{0\}~\text{such that, up to translations,}~u_n \rightharpoonup u_\infty~\text{as}~n\to\infty;
\end{multline}
\begin{equation}\label{proofs:eqn:I-1}
	\text{if}~0<\mu<\rho,~\text{then}~I_\rho\leq I_\mu+I_{\sqrt{\rho^2-\mu^2}};
\end{equation}
\begin{equation}\label{proofs:eqn:I-2}
	-\infty<I_\rho<0
	~\text{for all}~
	\rho>0;
\end{equation}
\begin{equation}\label{proofs:eqn:I-3}
	[0,\infty[ \ni \rho \mapsto I_\rho
	~\text{is continuous}
\end{equation}
and
\begin{equation}\label{proofs:eqn:I-4}
	I_\rho/\rho^2 \to 0
	~\text{as}~
	\rho\to 0.
\end{equation}
Suppose further that $T\colon H^1(\real^3)\to\real$ is a functional of class $C^1$ and if $\rho>0$ and $(u_n)_{n \in \nat}$ is a minimizing sequence for $I_\rho$, then
\begin{equation}\label{thm:eqn:2.1a}
	T(u_n-u_\infty)+T(u_\infty)-T(u_n) \to 0
	~\text{as}~
	n\to\infty;
\end{equation}
\begin{equation}\label{thm:eqn:2.1b}
	T(\alpha_n(u_n-u_\infty))-T(u_n-u_\infty) \to 0
	~\text{as}~
	n\to\infty,
\end{equation}
where $\alpha_n:=\|u_n-u_\infty\|_{L^2}^{-1}\sqrt{\rho^2-\|u_\infty\|_{L^2}^2}$ for every $n \in \nat$;
\begin{equation}\label{thm:eqn:2.1c}
	\limsup_{n \to \infty} \mathrm{d} T_{u_n}u_n<\infty;
\end{equation}
\begin{equation}\label{thm:eqn:2.1d}
	(\mathrm{d} T_{u_n}-\mathrm{d} T_{u_m})(u_n-u_m)\to 0
	~\text{as}~
	n,m\to\infty.
\end{equation}

We conclude that given $\rho>0$,
\[
	M(\rho):=\bigcup_{\mu \in ]0,\rho]}
	\{u \in S_\mu: I(u)=I_\mu\}
	\neq \emptyset.
\]
If we suppose further that
\begin{multline}\label{proofs:eqn:I-5}
	\text{given}~
	u\in M(\rho),
	~\text{there exists}~
	\beta\in\real
	~\text{such that}~
	\\
	]0,\infty[ \ni \theta \mapsto h^u_\beta(\theta):=I(u_{\beta,\theta})-\theta^2I(u)
	~\text{is differentiable and}~(h^u_\beta)'(1)\neq 0;
\end{multline}
then it holds that
\begin{equation}\tag{MD}\label{proof:eqn:MD}
	]0,\infty[ \ni \rho \mapsto I_\rho/\rho^2~\text{is strictly decreasing}.
\end{equation}
Moreover, if $(u_n)_{n\in\nat}$ is a minimizing sequence for $I_\rho$ and $u_n \rightharpoonup u_\infty\not\equiv 0$ as $n\to\infty$, then $I(u_\infty)=I_\rho$ and $\|u_n-u_\infty\|_{H^1}\to 0$ as $n \to\infty$.
\end{thm}

We remark that \cite[Theorem 2.1]{Bellazzini_2011} is more general than Theorem \ref{proofs:thm:2.1}. For instance, \cite[Theorem 2.1]{Bellazzini_2011} is also valid for more general Hilbert spaces (e. g., obtained as completions of $C_c^\infty(\real^3)$) and one may take paths of functions more general than those of the form $]0,\infty[\ni\theta\mapsto u_{\beta,\theta} \in H^1(\real^3)$ in \eqref{proofs:eqn:I-5} (see \cite[Definitions 2.1, 2.2]{Bellazzini_2011}). The second result we need is \cite[Lemma 2.1]{Bellazzini_2010}:
\begin{lem}\label{proofs:lem:2.1}
Let $\rho>0$, $\{u_n\}_{n\in\nat}\subset S_\rho$ be a minimizing sequence for $I_\rho$ such that $u_n \rightharpoonup u_\infty\not\equiv 0$ as $n\to\infty$ and let $\mu=\|u_\infty\|_{L^2} \in ]0,\rho]$. Suppose that $T\colon H^1(\real^3)\to\real$ is a functional of class $C^1$ such that \eqref{thm:eqn:2.1a}, \eqref{thm:eqn:2.1b} hold and suppose that
\begin{equation}\label{proofs:eqn:I-6}
	\text{if}~0<\mu<\rho,~\text{then}~I_\rho<I_\mu+I_{\sqrt{\rho^2-\mu^2}}.
\end{equation}
We conclude that $u_\infty \in S_\rho$. If we suppose further that \eqref{thm:eqn:2.1c}, \eqref{thm:eqn:2.1d} also hold, then $\|u_n-u_\infty\|_{H^1}\to 0$ as $n\to\infty$.
\end{lem}

\begin{rmk}
It is well known that \eqref{proof:eqn:MD} implies the \emph{strong subadditivity inequality} \eqref{proofs:eqn:I-6} (see Proof of Proposition \ref{proofB:prop}).
Moreover, \eqref{proofs:eqn:I-6} is equivalent to the condition that any minimizing sequence for $\left.I\right|_{S_\rho}$ is precompact (see \cite[Section 1.1]{Bellazzini_2011}).
\end{rmk}

Having introduced the abstract framework for minimization problems, let us state a few results which help us see how it fits the problem presented in the introduction.

\begin{lem}\label{abstract:lem:1}
If $a>0$ and we take $I=\mathcal{J}_a\colon H^1(\real^3)\to\real$, then conditions \eqref{proofs:eqn:I-0}, \eqref{proofs:eqn:I-1}, \eqref{proofs:eqn:I-3} and \eqref{proofs:eqn:I-4} are satisfied.
\end{lem}
\begin{proof}
Condition \eqref{proofs:eqn:I-0} follows from \cite[Proof of Theorem 4.1, Step 3]{Bellazzini_2011}. It suffices to argue as in \cite[Section I.2]{Lions_1984} to prove that \eqref{proofs:eqn:I-1} holds. Conditions \eqref{proofs:eqn:I-3}, \eqref{proofs:eqn:I-4} follow from \cite[Proof of Theorem 4.1, Step 4]{Bellazzini_2011}.
\end{proof}

It suffices to argue as in \cite[Proposition 3.1]{Bellazzini_2010} to prove the following lemma:
\begin{lem}\label{abstract:lem:2}
If $a>0$ and we take $T\colon H^1(\real^3)\to\real$ as given by
\[
	T(u)=\frac{1}{4}\int \phi_a^u u^2-\frac{1}{p}\|u\|_{L^p}^p,
\]
then conditions \eqref{thm:eqn:2.1a}--\eqref{thm:eqn:2.1d} are satisfied.
\end{lem}

In the last result of this section, we argue as Bellazzini and Siciliano in \cite[p. 275--6]{Bellazzini_2010} to prove that minimizing sequences for $\left.\mathcal{J}_a\right|_{S_\rho}$ do not converge in the weak sense to $0$:
\begin{lem}\label{abstract:lem:3}
Suppose that $a\geq 0$, $\rho>0$ and $\{u_n\}_{n\in\nat}\subset S_\rho$ is a minimizing sequence for $I:=\mathcal{J}_a$ such that $u_n\rightharpoonup u_\infty$ in $H^1(\real^3)$ as $n\to\infty$. We conclude that $u_\infty\not\equiv 0$.
\end{lem}
\begin{proof}
If we suppose that
\begin{equation}\label{abstract:eqn}
	\lim_{n\to\infty}\left(
		\sup_{y\in\real^3} \int_{B_1(y)} |u_n|^2
	\right)=0,
\end{equation}
then we conclude that $u_n\to 0$ in $L^p(\real^3)$ as $n\to\infty$. Due to the definition of $I$ and the fact that Lemma \ref{charac:lem} implies $I_\rho<0$, we conclude that this contradicts the fact that $(u_n)_{n\in\nat}$ is a minimizing sequence for $\left.I\right|_{S_\rho}$.

As \eqref{abstract:eqn} does not hold, we obtain that there exist $\mu>0$ and $\{y_n\}_{n\in\nat}\subset\real^3$ such that
\[
	\int_{B_1(0)} |u_n(\cdot+y_n)|^2\geq\mu
\]
whenever $n$ is large enough. The Sobolev embedding $H^1(B_1(0))\hookrightarrow L^2(B_1(0))$ is compact, so we conclude that $u_\infty\not\equiv 0$.
\end{proof}

\section{Proofs of main results}

\subsection{Proof of Theorem \ref{intro:theorem:small_rho}}

It suffices to verify that the conditions of Theorem \ref{proofs:thm:2.1} are satisfied. Due to Lemmas \ref{abstract:lem:1} and \ref{abstract:lem:2}, we only need to verify that conditions \eqref{proofs:eqn:I-2} and \eqref{proofs:eqn:I-5} hold in the case $I=\mathcal{J}_a$ for an arbitrary $a>0$. This is done in the following lemmas:
\begin{lem}
If $a>0$ and $I=\mathcal{J}_a\colon H^1(\real^3)\to\real$, then condition \eqref{proofs:eqn:I-2} holds.
\end{lem}
\begin{proof}
The fact that $I_\rho>-\infty$ for any $\rho>0$ follows from Lemma \ref{charac:lem}.

Let us prove that there exists $\beta\in\real$ such that
\begin{equation}\label{proofA:eqn:1}
	0<
	\left(
		1-\frac{3}{2}\beta
	\right)p
	+
	3\beta
	<
	\min(4-\beta,2-2\beta).
\end{equation}
First, suppose that $2<p<8/3$. In this case, $\beta \in \real$ satisfies the previous inequality precisely when
\[
	\beta
	<
	\min
	\left(
		\frac{8-2p}{8-3p}
		,
		\frac{2p}{3p-6}
	\right).
\]
Now, suppose that $p=8/3$. In this case, $\beta \in \real$ satisfies \eqref{proofA:eqn:1} precisely when
\[
	\beta<\frac{4-2p}{10-3p}.
\]
Finally, suppose that $8/3<p<3$. We have that $\beta \in \real$ satisfies \eqref{proofA:eqn:1} if, and only if,
\[
	\frac{2p-8}{3p-8}<\beta<\frac{4-2p}{10-3p}.
\]
There exists one such $\beta\in\real$ precisely when $8/3<p<3$ or $p>10/3$, hence the conclusion.

Fix $u\in S_1$ and $\beta\in\real$ which satisfies \eqref{proofA:eqn:1}. Due to \eqref{charac:ineq} and Lemma \ref{prior:lem:rescalings}, we conclude that
\begin{multline}\label{proofB:eqn:A}
	\text{given}~a>0~\text{and}~\rho\in ]0,1],
	\\
	\mathcal{J}_a(u_{\beta,\rho})
	\leq
	\mathcal{J}_0(u_{\beta,\rho})
	\leq
	\rho^{(1-3\beta/2)p+3\beta}
	\left[
		\rho^\alpha
		\left(
		\frac{1}{2}	\|\nabla u\|_{L^2}^2
		+
		\frac{1}{4}\int \phi_0^uu^2
		\right)
		-
		\frac{1}{p}\|u\|_{L^p}^p
	\right],
\end{multline}
where
\[
	\alpha:=
	\min(4-\beta,2-2\beta)
	-
	\left[
		\left(1-\frac{3}{2}\beta\right)p
		+
		3p
	\right]
	>0.
\]
It follows from \eqref{proofB:eqn:A} that there exists $R>0$ such that if $0<\rho<R$, then
\[
	\mathcal{J}_{a,\rho}
	\leq
	\mathcal{J}_{0,\rho}
	\leq
	\mathcal{J}_0(u_{\beta,\rho})
	<0.
\]

In order to conclude, it suffices to note that it follows from \eqref{proofs:eqn:I-1} that given $\rho>0$ such that $0<\mu<\rho~\text{implies}~\mathcal{J}_{a,\mu}<0$, it holds that given $\alpha\in [1,\sqrt{2}[$, $0<\mu<\rho~\text{implies}~\mathcal{J}_{a,\alpha\mu}<0$.
\end{proof}

\begin{lem}
If $a>0$ and $I=\mathcal{J}_a\colon H^1(\real^3)\to\real$, then condition \eqref{proofs:eqn:I-5} is satisfied.
\end{lem}
\begin{proof}
Suppose that \eqref{proofs:eqn:I-5} is not satisfied. In particular, we can fix $\{u_n\}_{n\in\nat}\subset M(\rho)$ such that $\|u_n\|_{L^2} \to 0$ as $n\to\infty$ and $(h^{u_n}_\beta)'(1)=0$ for every $\beta\in\real$ and $n \in\nat$.

Fix $n \in \nat$. We know that $(h^{u_n}_{2/3})'(1)=0$, so Lemma \ref{preliminary_results:lem:2} implies
\[
	\|\nabla u_n\|_{L^2}^2
	=
	\frac{1}{2}
	\left\{
		\frac{1}{2a}
		\int \int e^{-|y-x|/a}u_n(y)^2u_n(x)^2 \mathrm{d} x \mathrm{d} y
		+
		\int \phi^{u_n}_au_n^2
	\right\}.
\]
If we consider the previous equality and the elementary inequality
\[
	te^{-t/a}/a \leq (1-e^{-t/a})
	~\text{for}~
	t\geq 0,
\]
then we obtain
\begin{equation}\label{proofA:eqn:4.9-2}
	0
	\leq
	\|\nabla u_n\|_{L^2}^2
	\leq
	\frac{1}{4a}
	(1+2a)
	\int \phi^{u_n}_au_n^2.
\end{equation}
We know that $(h^{u_n}_0)'(1)=0$, so Lemma \ref{preliminary_results:lem:2} also implies
\begin{equation}\label{proofA:eqn:4.9-1}
	\frac{1}{p}\|u_n\|_{L^p}^p
	=
	\frac{1}{4}\frac{2}{p-2}\int \phi^{u_n}_au_n^2.
\end{equation}

As $\kappa_a\leq 1/a$, it follows from Young's inequality (\cite[Theorem 4.2]{Lieb_Loss_2001}) that
\[
	0
	\leq
	\int \phi^{u_n}_au_n^2
	\leq
	\frac{1}{a}\|u_n\|^4_{L^2}
	\to 0
	~\text{as}~
	n\to\infty.
\]
In view of \eqref{proofA:eqn:4.9-2} and \eqref{proofA:eqn:4.9-1}, we deduce that
\[
	\|\nabla u_n\|_{L^2},\|u_n\|_{L^p}\to 0
	~\text{as}~
	n\to\infty.
\]
By continuity, we conclude that $I(u_n)\to 0$ as $n\to\infty$. To summarize, we know that
\begin{equation}\label{proofA:eqn:analogous}
	I(u_n),\|\nabla u_n\|_{L^2},\|u_n\|_{L^p},\int \phi_a^{u_n}u_n^2\to 0~\text{as}~n\to\infty.
\end{equation}

The previous expression is analogous to \cite[(4.10)]{Bellazzini_2011}, so we can continue the argument as in Bellazzini and Siciliano's article.
Let us finish the proof in the case $2<p<12/5$ because similar arguments suffice for the other cases (see \cite[p. 2502--3]{Bellazzini_2011}). It follows from \eqref{energy-functional:eqn} and the Gagliardo--Nirenberg interpolation inequality (\cite[Theorem 12.83]{Leoni_2009}) that there exists $C_1>0$ such that
\[
	\int \phi_a^{u_n}u_n^2
	\leq
	\int \phi_0^{u_n}u_n^2
	\leq
	K\|u_n\|_{L^{12/5}}^{4}
	\leq
	C_1
	\|u_n\|_{L^p}^{4\alpha}
	\|u_n\|_{L^6}^{4(1-\alpha)},
\]
for every $n\in\nat$, where $\alpha:=3p/[2(6-p)]$. Due to Sobolev embeddings and \eqref{proofA:eqn:4.9-2}, we conclude that there exists $C_2>0$ such that given $n\in\nat$,
\[
	\int \phi_a^{u_n}u_n^2
	\leq
	C_2\left(\int \phi_a^{u_n}u_n^2\right)^{4\alpha/p+4(1-\alpha)/2}.
\]
As $4\alpha/p+4(1-\alpha)/2>1$, we obtain a contradiction with \eqref{proofA:eqn:analogous}.
\end{proof}

\subsection{Proof of Theorem \ref{intro:theorem:large_rho}}

Our first goal is to establish the following analogue to \cite[Lemma 3.2]{Bellazzini_2010}:
\begin{prop}\label{proofB:prop}
Given $p \in ]3,10/3[$, there exists $R_p>0$ such that given $a>0$ and $\rho>R_p$,
\begin{enumerate}
\item \label{proofB:prop:i}
	$-\infty<\mathcal{J}_{a,\rho}<0$;
\item \label{proofB:prop:ii}
	$\mathcal{J}_{a,\rho}<\mathcal{J}_{a,\mu}+\mathcal{J}_{a,\sqrt{\rho^2-\mu^2}}$
whenever $0<\mu<\rho$.
\end{enumerate}
\end{prop}
\begin{proof}
Let us prove item \ref{proofB:prop:i}. The fact that $\mathcal{J}_{a,\rho}>-\infty$ for any $\rho>0$ follows from Lemma \ref{charac:lem}. It is easy to check that
\[
	(1-3\beta/2)p+3\beta>\max(4-\beta,2-2\beta,0)
\]
if, and only if,
\[
	\frac{4-2p}{10-3p}<\beta<\frac{2p-8}{3p-8}.
\]
We can fix one such $\beta\in\real$ due to the hypothesis $3<p<10/3$. Fix $u\in S_1$. Due to \eqref{charac:ineq} and Lemma \ref{prior:lem:rescalings}, we conclude that
\begin{multline}\label{proofA:eqn:A}
	\text{given}~a>0~\text{and}~\rho\in [1,\infty[,
	\\
	\mathcal{J}_a(u_{\beta,\rho})
	\leq
	\mathcal{J}_0(u_{\beta,\rho})
	\leq
	\rho^{\max(4-\beta,2-2\beta)}
	\left[
		\frac{1}{2}	\|\nabla u\|_{L^2}^2
		+
		\frac{1}{4}\int \phi_0^uu^2
		-
		\frac{\rho^\alpha}{p}\|u\|_{L^p}^p
	\right],
\end{multline}
where
\[
	\alpha:=
	\left(1-\frac{3}{2}\beta\right)p
	+
	3p
	-
	\max(4-\beta,2-2\beta)
	>0.
\]
It follows from \eqref{proofA:eqn:A} that there exists $R_p>0$ such that if $\rho>R_p$, then given $a>0$,
\[
	\mathcal{J}_{a,\rho}
	\leq
	\mathcal{J}_{0,\rho}
	\leq
	\mathcal{J}_0(u_{\beta,\rho})<0.
\]
The previous inequality concludes the proof.

Now, let us prove item \ref{proofB:prop:ii}. It suffices to prove that there exists $R_p>0$ such that
\begin{equation}\label{proofB:eqn:aux}
	]R_p,\infty[ \ni \rho \mapsto \mathcal{J}_{a,\rho}/\rho^2~\text{is strictly decreasing}.
\end{equation}
Indeed, the previous assertion implies
\[
	\mathcal{J}_{a,\rho}
	=
	\frac{\mu^2}{\rho^2}\mathcal{J}_{a,\rho}
	+
	\frac{\rho^2-\mu^2}{\rho^2}\mathcal{J}_{a,\rho}
	<
	\mathcal{J}_{a,\mu}
	+
	\mathcal{J}_{a,\sqrt{\rho^2-\mu^2}}
\]
whenever $\rho>R_p$ and $0<\mu<\rho$. In this situation, let us prove that \eqref{proofB:eqn:aux} holds. Take $R_p>0$ as provided by item \ref{proofB:prop:ii}. Note that it suffices to prove that given $\rho>R_p$,
\[
	\od{}{\theta}\left.\left(
		\frac{\mathcal{J}_{a,\theta\rho}}{\theta^2}
	\right)\right|_{\theta=1}<0.
\]
We remark that given $u \in H^1(\real^3)$ and $\beta\in\real$, it holds that
\[
	(h^u_\beta)'(1)=
	\od{}{\theta}\left.\left[
		\frac{\mathcal{J}_a(u_{\beta,\theta})}{\theta^2}
	\right]\right|_{\theta=1}.
\]
Therefore, it suffices to prove that if $\rho>R_p$, then there exist $u \in S_\rho$ and $\beta\in\real$ such that $(h^u_\beta)'(1)<0$.

We know that $-\infty<\mathcal{J}_{a,\rho}<0$, so we can fix $u\in S_\rho$ such that $\mathcal{J}_a(u)<0$. In particular, it holds that
\begin{equation}\label{proofB:eqn:2}
	-\frac{1}{p}\|u\|_{L^p}^p
	<
	-
	\left(
		\frac{1}{2}\|\nabla u\|_{L^2}^2
		+
		\frac{1}{4}\int \phi_a^u u^2
	\right).
\end{equation}
Suppose that $\beta<0$. We have
\[
	\left[
		\left(1-\frac{3}{2}\beta\right)p+3\beta-2
	\right]>0,
\]
so we can use Lemma \ref{preliminary_results:lem:2} and \eqref{proofB:eqn:2} to obtain
\begin{multline*}
	(h^u_\beta)'(1)
	<
	-\frac{1}{2}
	\left[
		\left(
			1-\frac{3}{2}\beta
		\right)
		p+5\beta-2
	\right]
	\|\nabla u\|_{L^2}^2
	+
	\\
	-
	\frac{1}{4}
	\left[
		\left(
			1-\frac{3}{2}\beta
		\right)
		p+4\beta-4
	\right]
	\int \phi_a^u u^2.
\end{multline*}

There exists $\beta<0$ such that
\[
	\left[\left(
		1-\frac{3}{2}\beta
	\right)p+5\beta-2\right],
	\left[\left(
		1-\frac{3}{2}\beta
	\right)p+4\beta-4\right]
	>0
\]
if, and only if,
\[
	\frac{4-2p}{10-3p}
	<
	0
	~\text{and}~
	\frac{4-2p}{10-3p}
	<
	\frac{2p-8}{3p-8}.
\]
We remark that there exists $\beta<0$ which satisfies the previous condition precisely when $3<p<10/3$.

Due to the previous paragraphs, we can fix $\beta<0$ for which there exists $C_\beta>0$ such that if $\mathcal{J}_a(v)<0$, then $(h^v_\beta)'(1)\leq -C_\beta\|\nabla v\|_{L^2}^2$.

Fix $\{v_n\}_{n\in\nat}\subset S_\rho$ such that $\mathcal{J}_a(v_n)\to \mathcal{J}_{a,\rho}<0$ as $n\to\infty$. Let us prove that, up to subsequence, there exists $k>0$ such that $\|\nabla v_n\|_{L^2}>k$ for every $n\in\nat$. Indeed, suppose otherwise. Up to subsequence, we can suppose that $\|\nabla v_n\|_{L^2} \to 0$ as $n\to\infty$. The Gagliardo--Nirenberg interpolation inequality implies $\|v_n\|_{L^p} \to 0$ as $n\to\infty$, so we conclude that $\mathcal{J}_{a,\rho}=0$, which contradicts our hypothesis.

We conclude that if $\theta>1$ is sufficiently close to $1$, then
\[
	\frac{\mathcal{J}_{a,\theta\rho}}{\theta^2}
	\leq
	\mathcal{J}_{a,\rho}
	-
	\frac{1}{2}(\theta-1)C_\beta k.
\]
Finally, \eqref{proofB:eqn:aux} follows from the previous inequality. 
\end{proof}

Now, we will use Proposition \ref{proofB:prop} to prove Theorem \ref{intro:theorem:large_rho}:

\begin{proof}[Proof of Theorem \ref{intro:theorem:large_rho}]
Let $R_p>0$ be given by Proposition \ref{proofB:prop} and fix $\rho>R_p$. Let $I=\mathcal{J}_a\colon H^1(\real^3)\to\real$ and $(u_n)_{n\in\nat}$ be a minimizing sequence for $\left.I\right|_{S_\rho}$. Due to Lemma \ref{charac:lem}, $\{u_n\}_{n\in\nat}$ is a bounded subset of $H^1(\real^3)$. It follows from Lemma \ref{abstract:lem:3} that there exists $u_\infty\in H^1(\real^3)\setminus\{0\}$ such that $\|u_\infty\|_{L^2}\leq\rho$ and, up to subsequence and translations, $u_n\rightharpoonup u_\infty$ as $n\to\infty$. Due to Lemma \ref{abstract:lem:2}, we know that \eqref{thm:eqn:2.1a}--\eqref{thm:eqn:2.1d} are satisfied. Proposition \ref{proofB:prop} guarantees that \eqref{proofs:eqn:I-6} holds. In this situation, we can apply Lemma \ref{proofs:lem:2.1} to conclude that $u_\infty\in S_\rho$ and $I(u_\infty)=I_\rho$.
\end{proof}

\subsection{Proof of Theorem \ref{intro:theorem:radiality}}

The theorem will follow from a few adjustments in the arguments of \cite{Georgiev_2012}. In this context, we employ Georgiev, Prinari and Visciglia’s notation in this section for the convenience of the reader. Consider the following adjustments:
\begin{enumerate}
\item suppose that $2<p<14/5$ instead of $2<p<3$;
\item \cite[Remark 0.1]{Georgiev_2012} can be reformulated in the present context as the existence of least energy solutions to \eqref{intro:eqn:SBP_a,rho} (which is guaranteed by Theorem \ref{intro:theorem:small_rho});
\item one should take $\alpha(p)=(28-10p)/(10-3p)$ (which is positive precisely when $p<14/5$);
\item given $\rho\geq 0$, one should take
\[
	J_{\rho,p}
	:=
	\inf_{u \in S_1}
	\left\{
	\frac{1}{2}\|\nabla u\|_{L^2}^2
	+
	\rho^{\alpha(p)}
	\frac{1}{4}
	\int
		[u^2*\kappa_a(\theta^\beta \cdot)](x)u(x)^2
	\mathrm{d} x
	-
	\frac{1}{p}\|u\|_{L^p}^p
	\right\}
\]
and $\mathcal{E}_{\rho,p}\colon H^1(\real^3)\to\real$ should be defined as
\[
	\mathcal{E}_{\rho,p}(u)
	=
	\frac{1}{2}\|\nabla u\|_{L^2}^2
	+
	\rho
	\frac{1}{4}
	\int
		[u^2*\kappa_a(\theta^\beta \cdot)](x)u(x)^2
	\mathrm{d} x
	-
	\frac{1}{p}\|u\|_{L^p}^p,
\]
where
\[
	\beta:=-\frac{2p-4}{4-3(p-2)};
\]
\item mentions to the problem
\[
	-\Delta v + \omega v + \rho (v^2 * |\cdot|^{-1})-v|v|^{p-2}=0
\]
should be replaced by mentions to
\[
	-\Delta v + \omega v + \rho \phi_a^v v-v|v|^{p-2}=0.
\]
\end{enumerate}

\subsection{Proof of Theorem \ref{intro:theorem:a-to-zero:1}}
In order to prove Theorem \ref{intro:theorem:a-to-zero:1}, we need two preliminary lemmas. We begin by proving that, under certain hypotheses, $\mathcal{J}_{a,\rho}\to \mathcal{J}_{0,\rho}$ as $a\to 0$:
\begin{lem}\label{a-to-zero:lem:least-energy}
Consider $\rho>0$ and a family of minimizers
\[
	\{u_a \in S_\rho: \mathcal{J}_a(u_a)=\mathcal{J}_{a,\rho}~\text{and}~a>0\}.
\]
We conclude that $\mathcal{J}_{a,\rho}\to\mathcal{J}_{0,\rho}$ as $a\to 0$.
\end{lem}
\begin{proof}
It follows from \eqref{charac:ineq} that $\mathcal{J}_{a,\rho}\leq\mathcal{J}_{0,\rho}$ for any $a\geq 0$, so it suffices to prove that $\mathcal{J}_{0,\rho}\leq\liminf_{a\to 0}\mathcal{J}_{a,\rho}$. Given $a>0$, we have
\[
	\mathcal{J}_{0,\rho}\leq\mathcal{J}_0(u_a)=\mathcal{J}_{a,\rho}
	+
	\underbrace{
		\int \int \frac{e^{-|x-y|/a}}{|x-y|}u_a(x)^2u_a(y)^2 \mathrm{d} x \mathrm{d} y
	}_{=:f(a)}.
\]
In this situation, we only have to prove that $f(a)\to 0$ as $a\to 0$.

Fix $a \in ]0,1[$. It follows from a change of variable that
\[
	f(a)
	=
	\int
	\frac{e^{-|z|/a}}{|z|}
	\int
		u_a(x)^2u_a(x-z)^2
	\mathrm{d} x
	\mathrm{d} z.
\]
Clearly, 
\begin{multline*}
	f(a)
	=
	\underbrace{
		\int_{\{z\in\real^3: |z|<a|{\log a}|\}}
			\frac{e^{-|z|/a}}{|z|}
			\int
				u_a(x)^2u_a(x-z)^2
			\mathrm{d} x
		\mathrm{d} z
	}_{=:g(a)}
	+
	\\
	+
	\underbrace{
		\int_{\{z\in\real^3: |z|\geq a|{\log a}|\}}
			\frac{e^{-|z|/a}}{|z|}
			\int
				u_a(x)^2u_a(x-z)^2
			\mathrm{d} x
		\mathrm{d} z
	}_{=:h(a)}.
\end{multline*}
Let us bound $g(a)$. We know that $u_a \in L^{8/3}(\real^3)$, so it follows from Young's inequality that
\[
	\real^3 \ni z \mapsto \int u_a(x)^2u_a(x-z)^2 \mathrm{d} x
\]
is in $L^2(\real^3)$. More precisely, 
\begin{equation}\label{eqn:temporary:1}
	\left\|
		\int u_a(x)^2u_a(x-\cdot)^2 \mathrm{d} x
	\right\|_{L^2}
	\leq
	\|u_a\|^4_{L^{8/3}}.
\end{equation}
It is clear that
\[
	\real^3 \ni z
	\mapsto
	\frac{1}{|z|}\chi_{\{z\in\real^3: |z|<a|{\log a}|\}}
\]
is in $L^2(\real^3)$. More precisely,
\[
	\int_{\{z\in\real^3: |z|<a|{\log a}|\}}
	\frac{1}{|z|^2}
	\mathrm{d} z
	=4\pi a|{\log a}|.
\]
Considering \eqref{eqn:temporary:1} and the previous expression, we can use the Cauchy--Schwarz inequality to obtain
\begin{equation}\label{eqn:temporary:2}
	g(a)\leq 4\pi a|{\log a}|\|u_a\|^4_{L^{8/3}}.
\end{equation}
Let us bound $h(a)$. As $e^{-t/a} \leq a$ for $t\geq a|{\log a}|$, we can use assertion \eqref{energy-functional:eqn} to conclude that
\begin{equation}\label{eqn:temporary:3}
	h(a)
	\leq
	K a \|u_a\|^4_{L^{12/5}}.
\end{equation}

The bounds \eqref{eqn:temporary:2}, \eqref{eqn:temporary:3} were obtained for an arbitrary $a\in ]0,1[$ so we only need to prove that
\[
	\{\|u_a\|_{L^{8/3}}:0<a<1\},
	\{\|u_a\|_{L^{12/5}}:0<a<1\}
\]
are bounded to conclude by considering the respective limits as $a \to 0$. Indeed, it follows from Lemma \ref{charac:lem} that $\{\|u_a\|_{H^1}:a>0\}$ is bounded. Finally, the Sobolev embeddings $H^1(\real^3)\hookrightarrow L^{8/3}(\real^3), L^{12/5}(\real^3)$ let us conclude.
\end{proof}

Our last preliminary lemma is a result in \cite{dAvenia_2019}, and we provide its proof for the reader's convenience:
\begin{lem}[{\cite[Lemma 6.1]{dAvenia_2019}}]\label{a-to-zero:lem}
Let $\{f_0\}\cup\{f_a: 0<a<1\}\subset L^{6/5}(\real^3)$. Let $\phi_0$ be the unique solution in $\mathcal{D}^{1,2}(\real^3)$ to the partial differential equation
\begin{equation}\label{a-to-zero:eqn:PDE}
	-\Delta \phi = f_0.
\end{equation}
Given $a\in]0,1[$, let $\phi_a$ be the unique solution in $\mathcal{X}(\real^3)$ to the partial differential equation
\begin{equation}\label{a-to-zero:eqn:PDE-with-a}
	-\Delta \phi + a^2\Delta^2 \phi = f_a.
\end{equation}
As $a\to 0$, we have:
\begin{enumerate}
\item
if $f_a\rightharpoonup f_0$ in $L^{6/5}(\real^3)$, then $\phi_a\rightharpoonup \phi_0$ in $\mathcal{D}^{1,2}(\real^3)$;
\item
if $f_a\to f_0$ in $L^{6/5}(\real^3)$, then $\phi_a\to\phi_0$ in $\mathcal{D}^{1,2}(\real^3)$ and $a\Delta \phi_a \to 0$ in $L^2(\real^3)$.
\end{enumerate}
\end{lem}
\begin{proof}
Let us prove that as $a\to 0$, if $f_a\rightharpoonup f_0$ in $L^{6/5}(\real^3)$, then $\phi_a\rightharpoonup \phi_0$ in $\mathcal{D}^{1,2}(\real^3)$. Due to the Sobolev embedding $\mathcal{D}^{1,2}(\real^3)\hookrightarrow L^6(\real^3)$, there exists $C>0$ such that
\begin{equation}\label{a-to-zero:eqn:Sobolev}
	\|u\|_{L^6}\leq C\|\nabla u\|_{L^2}~\text{for every}~u\in\mathcal{D}^{1,2}(\real^3).
\end{equation}

Fix $a\in ]0,1[$. It follows from the fact that $\phi_a$ is a solution to \eqref{a-to-zero:eqn:PDE-with-a} and from Hölder's inequality that
\[
	\|\nabla \phi_a\|_{L^2}^2
	+
	a^2\|\Delta \phi_a\|_{L^2}^2
	=
	\int f_a\phi_a
	\leq
	\|f_a\|_{L^{6/5}}\|\phi_a\|_{L^6}.
\]
By considering \eqref{a-to-zero:eqn:Sobolev} and the previous inequality, we conclude that
\begin{equation}\label{a-to-zero:eqn:1}
	\|\nabla \phi_a\|_{L^2}\leq C\|f_a\|_{L^{6/5}},~
	a\|\Delta \phi_a\|_{L^2}\leq C\|f_a\|_{L^{6/5}}.
\end{equation}

We know that $f_a\rightharpoonup f_0$ in $L^{6/5}(\real^3)$ as $a\to 0$, so $\{f_a: 0<a<1\}$ is a bounded subset of $L^{6/5}(\real^3)$. It follows from \eqref{a-to-zero:eqn:1} that $\{\phi_a: 0<a<1\}$ is a bounded subset of $\mathcal{D}^{1,2}(\real^3)$. In particular, we conclude that for any $\{a_n\}_{n\in\nat}\subset ]0,1[$, there exists $\phi_*\in\mathcal{D}^{1,2}(\real^3)$ such that, up to subsequence, $\phi_{a_n}\rightharpoonup\phi_*$ in $\mathcal{D}^{1,2}(\real^3)$ as $n\to~\infty$. In order to conclude, it suffices to prove that if $\{a_n\}_{n\in\nat}\subset ]0,1[$ satisfies $a_n\to 0$ as $n\to\infty$ and $\phi_{a_n}\rightharpoonup\phi_*$ in $\mathcal{D}^{1,2}(\real^3)$ as $n\to~\infty$, then $\phi_*=\phi_0$.

Fix $\psi\in C_c^\infty(\real^3)$ and $n\in \nat$. The function $\phi_{a_n}$ is a solution to \eqref{a-to-zero:eqn:PDE-with-a}, so
\begin{equation}\label{a-to-zero:eqn:2}
	\int \langle\nabla \phi_{a_n}, \nabla \psi\rangle + a_n^2\Delta\phi_{a_n}\Delta\psi
	=
	\int f_{a_n}\psi.
\end{equation}
It follows from \eqref{a-to-zero:eqn:1} and the equality above that
\[
	a_n^2 \int\Delta\phi_{a_n}\Delta\psi
	\leq
	a_n^2\|\Delta\phi_{a_n}\|_{L^2}\|\Delta\psi\|_{L^2}
	\leq
	Ca_n^2\|f_{a_n}\|_{L^{6/5}}\|\Delta\psi\|_{L^2}.
\]

Considering the previous paragraph and the fact that $\{f_a: 0<a<1\}$ is a bounded subset of $L^{6/5}(\real^3)$, we can take the limit as $n\to\infty$ to conclude that
\[
	\int \langle\nabla \phi_*, \nabla \psi\rangle
	=
	\int f_0\psi
	~\text{for any}~
	\psi\in C_c^\infty(\real^3).
\]
The function $\phi_0$ is the unique solution to $\eqref{a-to-zero:eqn:PDE}$ in $\mathcal{D}^{1,2}(\real^3)$, so the previous assertion implies $\phi_*=\phi_0$. At this point, we have proved the first assertion of the lemma.

Let us prove that as $a\to 0$, if $f_a\to f_0$ in $L^{6/5}(\real^3)$, then $\phi_a\to \phi_0$ in $\mathcal{D}^{1,2}(\real^3)$. We know that $\|\nabla \cdot\|_{L^2}$ is weakly lower semicontinuous and it follows from the previous item that $\phi_a\rightharpoonup\phi_0$ in $\mathcal{D}^{1,2}(\real^3)$ as $a\to 0$, so
\begin{equation}\label{a-to-zero:eqn:3}
	\|\nabla \phi_0\|_{L^2}^2
	\leq
	\liminf_{a\to 0}
	\|\nabla \phi_a\|_{L^2}^2.
\end{equation}
Fix $\{\psi_n\}_{n\in\nat}\subset C_c^\infty(\real^3)$ such that $\psi_n \to \phi_0$ in $\mathcal{D}^{1,2}(\real^3)$ as $n\to\infty$. 

Fix $a\in]0,1[$ and let $E_a\colon\mathcal{X}(\real^3)\to\real$ be given by
\[
	E_a(\phi)
	=
	\frac{1}{2}\|\nabla\phi\|_{L^2}^2
	+
	\frac{a^2}{2}\|\Delta\phi\|_{L^2}^2
	-
	\int f_a\phi.
\]
Due to the equality $E_a(\phi_a)=\min_{u\in\mathcal{X}(\real^3)} E_a(u)$, we obtain
\[
	\frac{1}{2}\|\nabla \phi_a\|_{L^2}^2
	\leq
	\frac{1}{2}\|\nabla \psi_n\|_{L^2}^2
	+
	\frac{a^2}{2}\|\Delta \psi_n\|_{L^2}^2
	-
	\int f_a(\psi_n-\phi_a)
	~\text{for every}~
	n\in\nat.
\]

The function $\phi_0$ solves \eqref{a-to-zero:eqn:PDE} and it follows from the previous assertion of the lemma that $\phi_a\rightharpoonup \phi_0$ in $\mathcal{D}^{1,2}(\real^3)$ as $a\to 0$, so $\int f_a\phi_a \to \int f_0\phi_0$ as $a \to 0$. It follows from the previous limit and the previous paragraph that given $n\in\nat$,
\[
	\limsup_{a\to 0}
	\frac{1}{2}\|\nabla \phi_a\|_{L^2}^2
	\leq
	\frac{1}{2}\|\nabla \psi_n\|_{L^2}^2
	-
	\int f_0(\psi_n-\phi_0).
\]
Passing to the limit $n\to\infty$, we obtain
\begin{equation}\label{a-to-zero:eqn:4}
	\limsup_{a\to 0} \frac{1}{2}\|\nabla \phi_a\|_{L^2}^2
	\leq
	\|\nabla \phi_0\|_{L^2}^2.
\end{equation}
It follows from \eqref{a-to-zero:eqn:3} and \eqref{a-to-zero:eqn:4} that $\|\nabla \phi_a\|_{L^2}^2\to\|\nabla \phi_0\|_{L^2}^2$ as $a\to 0$. We already knew that $\phi_a\rightharpoonup\phi_0$ in $\mathcal{D}^{1,2}(\real^3)$ as $a\to 0$, so we conclude that $\phi_a\to\phi_0$ in $\mathcal{D}^{1,2}(\real^3)$ as $a\to 0$.

Finally, the last limit in the statement of the lemma follows trivially.
\end{proof}

Having established these preliminary lemmas, we can finally prove the theorem:
\begin{proof}[Proof of Theorem \ref{intro:theorem:a-to-zero:1}]
Theorem \ref{intro:theorem:radiality} guarantees that there exists $R_p>0$ such that if $a>0$ and $0<\rho<R_p$, then, up to translations, \eqref{intro:eqn:a-to-zero} is a subset of $H^1_{\mathrm{rad}}(\real^3)\times\mathcal{X}_{\mathrm{rad}}(\real^3)\times\real$. Up to further decreasing $R_p>0$, it follows from Proposition \ref{multipliers:prop} that  given $\rho \in ]0,R_p[$,
\begin{multline}\label{a-to-zero:eqn:positive-multiplier}
	\text{if}~
	(u,\phi_0^u,\omega)\in H^1(\real^3)\times\mathcal{D}^{1,2}(\real^3)\times\real
	\\
	~\text{is a least energy solution to}~
	\eqref{intro:eqn:SPS_rho},
	~\text{then}~
	\omega>0.
\end{multline}

It follows from Lemmas \ref{charac:lem} and \ref{a-to-zero:lem:least-energy} that
\begin{equation}\label{a-to-zero:eqn:bounded}
	\{\|u_a\|_{H^1}: a>0\}~\text{is bounded}.
\end{equation}

Let us prove that
\begin{equation}\label{a-to-zero:eqn:bounded-multipliers}
	\{\omega_a: a>0\}~\text{is bounded}.
\end{equation}
It is clear that
\begin{equation}\label{a-to-zero:eqn:omega_a}
	\omega_a
	=
	\frac{1}{\rho^2}
	\left(
		\|u_a\|^p_{L^p}-\|\nabla u_a\|^2_{L^2}-\int \phi_a^{u_a} u_a^2
	\right)
	~\text{for every}~
	a\in ]0,1[.
\end{equation}
Due to this previous expression, \eqref{a-to-zero:eqn:bounded-multipliers} follows from \eqref{energy-functional:eqn}, \eqref{a-to-zero:eqn:bounded} and the Sobolev embeddings $H^1(\real^3)\hookrightarrow L^p(\real^3), L^{12/5}(\real^3)$.

Fix $\{a_n\}_{n\in\nat}\subset ]0,1[$ such that $a_n \to 0$ as $n\to\infty$. To clear up the notation, define $(u_n,\phi_n,\omega_n)=(u_{a_n},\phi_{a_n},\omega_{a_n})$ for every $n\in\nat$. Due to \eqref{a-to-zero:eqn:bounded}, there exists $u_0\in H^1_{\mathrm{rad}}(\real^3)$ such that, up to subsequence, $u_n\rightharpoonup u_0$ in $H^1_{\mathrm{rad}}(\real^3)$ as $n\to\infty$.

It follows from Lemma \ref{a-to-zero:lem:least-energy} that $(u_n)_{n\in\nat}$ is a minimizing sequence for $\left.\mathcal{J}_0\right|_{S_\rho}$, so we can use Lemma \ref{abstract:lem:3} to conclude that $u_0\not\equiv 0$. It was proved in \cite{Bellazzini_2010} that Lemma \ref{proofs:lem:2.1} may be applied in the case $a=0$, so we conclude that $u_0\in S_\rho$ is a minimizer of $I:=\left.\mathcal{J}_0\right|_{S_\rho}$ and $u_n\to u_0$ in $H^1_{\mathrm{rad}}(\real^3)$ as $n\to\infty$. In particular, we conclude that $(u_0,\phi_0,\omega_0)$ is a least energy solution to \eqref{intro:eqn:SPS_rho}, where $\phi_0$ is the unique function in $\mathcal{D}^{1,2}(\real^3)$ which satisfies
\[
	-\Delta \phi_0=4\pi u_0^2
\]
and
\[
	\omega_0
	:=
	\frac{1}{\rho^2}
	\left(
		\|u_0\|^p_{L^p}-\|\nabla u_0\|^2_{L^2}-\int \phi_0 u_0^2
	\right).
\]
Furthermore, implication \eqref{a-to-zero:eqn:positive-multiplier} shows that $\omega_0>0$.

Let us present an alternative proof (based on \cite[Proof of Theorem 1.3]{dAvenia_2019}) that $u_n\to u_0$ in $H^1_{\mathrm{rad}}(\real^3)$ as $n\to\infty$. Let $R\colon H^1_{\mathrm{rad}}(\real^3)\to H^{-1}_{\mathrm{rad}}(\real^3)$ be the Riesz isomorphism given by
\[
	Rv=\langle\nabla v,\nabla \cdot\rangle_{L^2}+\omega_0\langle v,\cdot\rangle_{L^2}.
\]
Given $n \in \nat$, it holds that
\[
	Ru_n+\omega_n\langle u_n,\cdot\rangle_{L^2}+\langle\phi_{a_n}^{u_n}u_n,\cdot\rangle_{L^2}
	=
	\langle u_n|u_n|^{p-2},\cdot\rangle_{L^2}+\omega_0\langle u_n,\cdot\rangle_{L^2},
\]
so
\begin{equation}\label{a-to-zero:eqn:RHS}
	u_n
	=
	-R^{-1}\langle\phi_{a_n}^{u_n}u_n,\cdot\rangle_{L^2}
	+
	R^{-1}\langle u_n|u_n|^{p-2},\cdot\rangle_{L^2}+(\omega_0-\omega_n)
	R^{-1}\langle u_n,\cdot\rangle_{L^2}.
\end{equation}
Due to \eqref{a-to-zero:eqn:bounded-multipliers}, we know that $(\omega_0-\omega_n)_{n\in\nat}$ converges up to subsequence. The Sobolev embeddings $H_{\mathrm{rad}}^1(\real^3)\hookrightarrow L^{12/5}(\real^3), L^p(\real^3)$ are compact so, up to subsequence, the right-hand side of \eqref{a-to-zero:eqn:RHS} converges as $n\to\infty$. We conclude that, up to subsequence, $u_n\to u_0$ in $H^1_{\mathrm{rad}}(\real^3)$ as $n\to\infty$.

The compact embedding $H_{\mathrm{rad}}^1(\real^3)\hookrightarrow L^{12/5}(\real^3)$ implies $u_n^2 \to u_0^2$ in $L^{6/5}(\real^3)$ as $n\to\infty$, so we can apply Lemma \ref{a-to-zero:lem} to conclude that $\phi_n \to \phi_0 \in H^1(\real^3)$ in $\mathcal{D}^{1,2}(\real^3)$ as $n\to\infty$.

To prove that $\omega_n \to \omega_0$ as $n\to\infty$, it suffices to consider \eqref{a-to-zero:eqn:omega_a} and take limits while considering the aforementioned Sobolev embeddings.
\end{proof}

\sloppy
\printbibliography
\end{document}